\documentclass[10pt,reqno]{amsart}
\usepackage{graphicx, amssymb, color,pdfsync}
\usepackage[all,cmtip]{xy}
\numberwithin{equation}{section}
\usepackage{mathrsfs}
\usepackage{diagbox}
\usepackage{tikz}
\usetikzlibrary{matrix,arrows}
\DeclareMathOperator{\Aut}{Aut}
\usepackage[switch]{lineno}

\begin{document}
\newcommand{\s}{\vspace{0.2cm}}

\newtheorem{theo}{Theorem}
\newtheorem{prop}{Proposition}
\newtheorem{coro}{Corollary}
\newtheorem{lemm}{Lemma}
\newtheorem{claim}{Claim}
\newtheorem{example}{Example}
\theoremstyle{remark}
\newtheorem{rema}{\bf Remark}
\newtheorem{defi}{\bf Definition}

\title[ Riemann surfaces of genus $q$ with $4q$ automorphisms]{On the one-dimensional family of Riemann surfaces of genus $q$ with $4q$ automorphisms}
\date{}

\author{Sebasti\'an Reyes-Carocca}
\address{Departamento de Matem\'atica y Estad\'istica, Universidad de La Frontera, Avenida Francisco Salazar 01145, Temuco, Chile.}
\email{sebastian.reyes@ufrontera.cl}

\thanks{Partially supported by Postdoctoral Fondecyt Grant 3160002, Redes Etapa Inicial Grant REDI-170071 and Anillo ACT1415 PIA  CONICYT Grant}

\begin{abstract} Bujalance, Costa and Izquierdo have recently proved that all those Riemann surfaces of genus $g \ge 2$ different from $3, 6, 12, 15$ and 30, with exactly $4g$ automorphisms form an equisymmetric one-dimensional family, denoted by $\mathcal{F}_g.$ In this paper, for every prime number $q \ge 5,$ we explore further properties of each Riemann surface $S$ in $\mathcal{F}_q$ as well as of its Jacobian variety $JS.$
\end{abstract}
\maketitle

\section{Introduction}Automorphism groups of compact Riemann surfaces have been extensively studied, going back to Wiman, Klein and Hurwitz, among others. 

It is classically known that the full automorphism group of a Riemann surface of genus $g \ge 2$ is finite; its size is bounded by $84(g-1).$ Moreover, there are infinitely many integers $g$ for which this bound is attained; see \cite{McB2}.

\s

Usually when additional conditions are imposed on a group of automorphisms, a smaller bound for its order is obtained; for example, classical results assert that in the abelian and cyclic case these bounds are $4g+4$ and $4g+2$ respectively. 

It is an interesting problem to understand the extent to which the order of the full automorphism group determines the Riemann surface; see for example \cite{K1}, \cite{K2} and \cite{Naka}.

\s

Very recently, Costa and Izquierdo have proved that the maximal order of the form $ag+b$ (for fixed integers $a$ and $b$) of the full automorphism group of equisymmetric and one-dimensional families of Riemann surfaces of genus $g \ge 2$ appearing in all genera is $4g+4.$ Moreover, they constructed explicit families attaining this bound; see \cite{CIn}. The second possible largest order is $4g.$

\s

Bujalance, Costa and Izquierdo have recently determined all those Riemann surfaces of genus $g \ge 2$ with exactly $4g$ automorphisms. More precisely, following \cite[Theorem 7]{costa}, if $g$ is different from the exceptional values $3, 6, 12, 15$ and 30, then the Riemann surfaces of genus $g$ admitting exactly $4g$ automorphisms form an equisymmetric one-dimensional family, denoted by $\mathcal{F}_g.$  Moreover, if $S$ is a Riemann surface in $\mathcal{F}_g$ then its full automorphism group $G$ is isomorphic to a dihedral group, and the corresponding quotient $S/G$ has genus zero.

\s

The present article is devoted to study further properties of each member $S$ of the family $\mathcal{F}_g$ and of its Jacobian variety $JS.$ In spite of the fact that the results of this paper might be stated for each integer $g \ge 2$ different from $3, 6, 12, 15$ and 30, for the sake of simplicity we shall restrict to the case $g=q \ge 5$ prime.

\s

This paper is organized as follows.

\s

In Section \ref{s2} we shall introduce the basic background; namely, group actions on Riemann surfaces, complex tori and abelian varieties, representation of groups, and the group algebra decomposition theorem for Jacobians.

\s

In Section \ref{s3} we shall take advantage of the hyperellipticity of the Riemann surfaces in the family $\mathcal{F}_q$ (see \cite[Remark 9]{costa}) to determine explicit algebraic descriptions of them. In addition, with respect to these models, we will provide realizations of their full automorphism groups. 

\s

 If a finite group $G$ acts on a Riemann surface $S,$ then this action induces a $G$-equivariant isogeny decomposition of $JS$ into a product of abelian subvarieties, of the form $JS \sim_{G} \Pi_{i=1}^rB_i^{n_i}.$ This decomposition (known as the {\it group algebra decomposition} of $JS$ with respect to $G;$ see \cite{cr} and \cite{l-r}) only depends on the algebraic structure of the group; however, further information such as the dimension of each factor $B_i$ depends on the geometry of the action. Following \cite{yoibero}, the dimension of each $B_i$ is explicitly given after choosing a {\it generating vector} of $G$ representing the action on $S.$ Because of this dependence, the general question arises about how the choice of such a generating vector affects the group algebra decomposition of $JS$.

In Section \ref{s4} we shall give a complete answer to the aforementioned question, for each Riemann surface $S$ in the family $\mathcal{F}_q$. To prove this result, we shall begin by proving some lemmata concerning the rational representations of $G,$ and the possible generating vectors representing the action. We shall also prove that each Jacobian $JS$ contains an elliptic curve, and that it decomposes into a product of Jacobians of quotients of $S.$

\s

In Section \ref{s5} we shall explore the fields of definition of the Riemann surfaces $S$ in the family $\mathcal{F}_q$. More precisely, we shall give a characterization for when $S$ and $JS$ can be defined, as projective algebraic varieties, by polynomials with real coefficients, and by polynomials with algebraic coefficients. Furthermore, in the latter case we prove that $JS$ decomposes in terms of abelian subvarieties which can also be defined by polynomials with algebraic coefficients. We shall also observe that $S$ and $JS$ can be defined over the field of moduli of $S$.

\s

 Finally, in Section \ref{s6} we shall compute the dimension of a {\it special variety} in the moduli space $\mathcal{A}_q$ of principally polarized abelian varieties of dimension $q$ associated to each Riemann surface $S$ in the family $\mathcal{F}_q,$ called the {\it Shimura family}; see \cite{Wolfart2}. Moreover, for the particular case $q=5,$ we will be able to describe its elements --by exhibiting period matrices-- as members of a three-dimensional family in $\mathcal{A}_5$ admitting a fixed action of the dihedral group of order 20.

\s

{\bf Acknowledgements.} The author wishes to express his gratitude to Anita M. Rojas for sharing her SAGE routines with him.

\section{Preliminaries} \label{s2}

\subsection{Group actions on Riemann surfaces}
Let $S$ be a compact Riemann surface and let $G$ be a finite group. We denote by $\mbox{Aut}(S)$ the full automorphism group of $S,$ and say that $G$ acts on $S$ if there is a group monomorphism $\psi: G\to \Aut(S).$ The space of orbits $S/G$ of the action of $G \cong \psi(G)$ on $S$ is naturally endowed with a Riemann surface structure in such a way that the natural projection $\pi : S \to S/G$ is holomorphic. The degree of $\pi$ is the order $|G|$ of $G$ and the multiplicity of $\pi$ at $p \in S$ is $|G_p|$, where $G_p$ denotes the stabilizer of $p$ in $G$. If $|G_p| \neq 1$ then $p$ is called a \textit{branch point} and its image by $\pi$ is a \textit{branch value}.

Let $\{p_1, \ldots, p_l\}$ be a maximal collection of non-$G$-equivalent branch points of $\pi.$ The {\it signature} of the action of $G$ on $S$ is the tuple $(\gamma; m_1, \ldots, m_l)$ where $\gamma$ is the genus of the quotient $S/G$ and $m_i=|G_{p_i}|.$ If $\gamma=0$ we write $(m_1, \ldots, m_r)$ for short. The branch value $\pi(p_i)$ is said to be {\it marked} with $m_i.$ The Riemann-Hurwitz formula relates these numbers, the order of $G$ and the genus $g$ of $S;$ namely \begin{equation*} \label{R-R}
2g-2=|G|(2 \gamma -2) + |G| \cdot \Sigma_{i=1}^l (1- \tfrac{1}{m_i}).
\end{equation*}

A $2\gamma+l$ tuple $(a_1, \ldots, a_{\gamma}, b_1, \ldots,
b_{\gamma}, c_1, \ldots,c_l)$ of elements of $G$ is called a
\textit{generating vector of $G$ of type $(\gamma; m_1, \ldots ,m_l)$} if the
following conditions are satisfied:

\begin{enumerate}
\item[(a)] $G$ is generated by the elements $a_1, \ldots, a_{\gamma}, b_1,
\ldots, b_{\gamma}, c_1, \ldots, c_l,$
\item[(b)] $\text{order}(c_i)=m_i$ for $1 \le i \le l,$ and
\item[(c)] $\Pi_{j=1}^{\gamma}[a_j, b_j] \Pi_{i=1}^l c_i=1$, where $[a_i,b_i]=a_ib_ia_i^{-1}b_i^{-1}.$
\end{enumerate}

Riemann's existence theorem ensures that the group $G$ acts on a Riemann surface of genus $g$ with signature $(\gamma; m_1, \ldots, m_l)$ if and only if the Riemann-Hurwitz formula is satisfied and $G$ has a generating vector of type $(\gamma; m_1, \ldots ,m_l).$ See \cite{Brou}.

If we denote by $C_j$ the conjugacy class of the subgroup $G_{p_j}$ in $G$ then, the tuple $(\gamma; [m_1, C_1], \ldots, [m_l, C_l])$ is called a {\it geometric signature} for the action of $G$ on $S$. This concept was introduced in \cite{yoibero} in order to control the behavior of the intermediate coverings ($S \to S/H$ for a subgroup $H$ of $G$) and the dimension of the factors arising in the group algebra decomposition of $JS$ (see Subsection \ref{GAD}). 

We shall say that the geometric signature $(\gamma; [m_1, C_1], \ldots, [m_l, C_l])$ is {\it associated} to the generating vector $(a_1, \ldots, a_{\gamma}, b_1, \ldots, b_{\gamma}, c_1, \ldots,c_l)$ because the subgroup of $G$ generated by $c_i$ is in the conjugacy class $C_i.$ 

\subsection{Topologically equivalent actions} Let $\text{Hom}^+(S)$ denote the group of orientation preserving homeomorphisms of $S.$ Two actions $\psi_1$ and $\psi_2$ of $G$ on $S$ are {\it topologically equivalent} if there exist $\omega \in \Aut(G)$ and $h \in \text{Hom}^+(S)$ such that
\begin{equation}\label{equivalentactions}
\psi_2(g) = h \psi_1(\omega(g)) h^{-1} \hspace{0.5 cm} \mbox{for all} \,\, g\in G.
\end{equation}

In terms of Fuchsian groups, the action of $G$ on $S$ can be constructed by means of a pair of Fuchsian groups $\Gamma \trianglelefteq  \Delta$ such that $S = \mathbb{H}/\Gamma$, with $\mathbb{H}$ denoting the upper half-plane, and an epimorphism $\theta: \Delta \to G$ with kernel $\Gamma$. The group $\Gamma$ is torsion-free and isomorphic to the fundamental group of $S.$ It is also known that $\Delta$ has a presentation given by generators $\alpha_1, \ldots, \alpha_{\gamma}, \beta_1, \ldots, \beta_{\gamma}, \gamma_1, \ldots , \gamma_l$ and relations
$$\gamma_1^{m_1}=\cdots =\gamma_l^{m_l}=\Pi_{j=1}^{\gamma}[\alpha_j, \beta_j] \Pi_{i=1}^l \gamma_i=1.$$

Note that there is a bijective correspondence between the set of generating vectors of $G$ of type $(\gamma, m_1, \ldots, m_l)$ and the set $\mathscr{K}$ of epimorphism of groups $\Delta \to G$ with torsion-free kernel.

Each orientation preserving homeomorphism $h$ satisfying (\ref{equivalentactions}) yields a group automorphism $h^*$ of $\Delta$; we denote by $\mathscr{B}$ the subgroup of $\mbox{Aut}(\Delta)$ consisting of them. The group $\mbox{Aut}(G) \times \mathscr{B}$ acts on $\mathscr{K}$ by $$((\omega, h^*), \theta) \mapsto \omega \circ \theta \circ (h^*)^{-1}$$ and therefore it also acts on the set of generating vectors of $G$ of type $(\gamma, m_1, \ldots, m_l).$

Let $\sigma_1$ and $\sigma_2$ be two generating vectors of type $(\gamma, m_1, \ldots, m_l)$ of $G.$ Then $\sigma_1$ and $\sigma_2$ define topologically equivalent actions if and only if $\sigma_1$ and $\sigma_2$ are in the same $(\mbox{Aut}(G) \times \mathscr{B})$-orbit (see \cite{Brou}; also \cite{Harvey} and \cite{McB}). 

\s

We refer to the classical articles \cite{singerman2} and \cite{singerman} for more details concerning the relationship between Riemann surfaces, generating vectors and Fuchsian groups.

\subsection{Abelian varieties} A $g$-dimensional {\it complex torus} $X=V/\Lambda$ is the quotient of a $g$-dimensional complex vector space $V$ by a lattice $\Lambda.$ Each complex torus is an abelian group and a $g$-dimensional compact connected complex analytic manifold.  

Complex tori can be described in a very concrete way, as follows. Choose bases \begin{equation} \label{bases}B_V=\{v_i\}_{i=1}^{g} \hspace{0.5 cm} \mbox{and} \hspace{0.5 cm} B_{\Lambda}=\{\lambda_j\}_{j=1}^{2g}\end{equation}of $V$ as a $\mathbb{C}$-vector space, and of $\Lambda$ as a $\mathbb{Z}$-module, respectively. Then there are complex constants $\{\pi_{ij}\}$ such that $\lambda_j=\Sigma_{i=1}^g \pi_{ij} v_i.$ The matrix $$\Pi=(\pi_{ij}) \in \mbox{M}(g \times 2g, \mathbb{C})$$represents $X,$ and is known as the {\it period matrix} for $X$ with respect to \eqref{bases}.

A {\it homomorphism} between  complex tori is a holomorphic map which is also a homomorphism of groups. We shall denote by $\mbox{End}(X)$ the ring of endomorphism of $X.$ An {\it automorphism} of a complex torus is a bijective homomorphism into itself.

Special homomorphisms are {\it isogenies:} these are surjective homomorphisms with finite kernel; isogenous tori are denoted by $X_1 \sim X_2$. The isogenies of a complex torus $X$ into itself are the invertible elements of the ring of rational endomorphisms  $$\mbox{End}_{\mathbb{Q}}(X):=\mbox{End(X)} \otimes_{\mathbb{Z}} \mathbb{Q}.$$

An {\it abelian variety} is a complex torus which is also a projective algebraic variety. Each abelian variety $X=V/\Lambda$ admits a {\it polarization}; namely, a non-degenerate real alternating form $\Theta$ on $V$ such that for all $v,w \in V$ $$\Theta(iv, iw)=\Theta(v,w) \hspace{0.5 cm} \mbox{and} \hspace{0.5 cm} \Theta(\Lambda \times \Lambda) \subset \mathbb{Z}.$$

If the elementary divisors of $\Theta|_{\Lambda \times \Lambda}$ are $\{1, \stackrel{g}{\ldots}, 1\},$ where $g$ is the dimension of $X,$ then
the polarization $\Theta$ is called {\it principal} and the pair $(X,\Theta)$ is called a 
{\it principally polarized abelian variety} (from now on we write {\it ppav} for short).

Let $(X=V/\Lambda, \Theta)$ be a ppav of dimension $g.$ Then there exists a basis for $\Lambda$ such that the matrix for $\Theta$ with respect to it is \begin{equation*}\label{simpl}
J = \left( \begin{smallmatrix}
0 & I_g \\
-I_g & 0
\end{smallmatrix} \right) 
\end{equation*}with $I_g$ denoting the $g \times g$ identity matrix; such a basis is termed {\it symplectic.} Furthermore, there exist a basis for $V$ and a symplectic basis for $\Lambda$ which respect to which the period matrix for $X$ is $\Pi=(I_g \, Z),$ where $Z$ belongs to the Siegel space $$\mathscr{H}_g=\{ Z \in \mbox{M}(g, \mathbb{C}) : Z = Z^t \, \mbox{and } \mbox{Im}(Z) >0\},$$with $Z^t$ denoting the transpose matrix of $Z.$

\subsection{Moduli space of ppavs} 

An {\it isomorphism} between ppavs is an isomorphism of the underlying complex torus structures that preserves the polarizations.

Let $(X_i, \Theta_i)$ be a ppav of dimension $g,$ and let $\Pi_i=(I_g \, Z_i)$ be the period matrix of $X_i$ with respect to chosen basis, for $i=1,2$. Each isomorphism between $(X_1, \Theta_1)$ and $(X_2, \Theta_2)$ is given by a pair of matrices $$M \in \mbox{GL}(g, \mathbb{C}) \hspace{0,4 cm} \mbox{and} \hspace{0,4 cm} R \in \mbox{GL}(2g, \mathbb{Z})$$ (corresponding to the {\it analytic} and {\it rational} representations, respectively) such that\begin{equation} \label{ig} M(I_g \, Z_1)=(I_g \, Z_2)R.\end{equation}

Since $R$ preserves the principal polarizations, it belongs to the symplectic group $$\mbox{Sp}(2g, \mathbb{Z})=\{ R \in  \mbox{M}(2g, \mathbb{Z}) : R^tJ R=J  \}.$$ 

Now, it follows from \eqref{ig} that the correspondence
\begin{equation*} \label{cuadrada} \mbox{Sp}(2g, \mathbb{Z}) \times \mathscr{H}_g \to \mathscr{H}_g \hspace{0.8 cm}    ( \left( \begin{smallmatrix}
A & B \\
C & D
\end{smallmatrix} \right) , Z ) \mapsto (A+ZC)^{-1}(B+ZD) 
\end{equation*}defines an action which identifies period matrices representing isomorphic ppavs. Hence, the quotient $$\mathscr{H}_g \to \mathcal{A}_g:=\mathscr{H}_g/ \mbox{Sp}(2g, \mathbb{Z})$$is the {\it moduli space} of isomorphism classes of ppavs of dimension $g.$

We refer to \cite{bl} and \cite{debarre} for more details on abelian varieties.

\subsection{Representations of groups} \label{rep} Let $G$ be a finite group and let $\rho : G \to \mbox{GL}(V)$ be a complex representation of $G.$ Abusing notation, we shall also write $V$ to refer to the representation $\rho.$ The {\it degree} $d_V$ of $V$ is the dimension of $V$ as complex vector space, and the character $\chi_V$ of $V$ is the map obtained by associating to each $g \in G$ the trace of the matrix $\rho(g).$ Two representations $V_1$ and $V_2$ are {\it equivalent} if and only if their characters agree; we write $V_1 \cong V_2.$ The {\it character field} $K_V$ of $V$ is the field obtained by extending the rational numbers by the values of the character $\chi_{V}$. The {\it Schur index} $s_V$ of $V$ is the smallest positive integer such that there exists a field extension $L_V$ of $K_V$ of degree $s_V$ over which $V$ can be defined. 

It is a known fact that for each rational irreducible representation $W$ of $G$ there is a complex irreducible representation $V$ of $G$ such that $$W \otimes_{\mathbb{Q}}\mathbb{C}  \cong   (\oplus_{\sigma} V^{\sigma}) \oplus \stackrel{s_V}{\cdots}\oplus (\oplus_{\sigma} V^{\sigma}) = s_V\, \left( \oplus_{\sigma} V^{\sigma} \right),$$where the sum $\oplus_{\sigma}$ is taken over the Galois group associated to the extension $\mathbb{Q} \le K_V.$ We say that $V$ is {\it associated} to $W.$ 

We refer to \cite{Serre} for further basic facts related to representations of groups.

\subsection{Group algebra decomposition theorem for Jacobians}\label{GAD}Let $S$ be a  Riemann surface of genus $g.$ Let us denote by $\mathscr{H}^{1,0}(S, \mathbb{C})$ the $g$-dimensional complex vector space of the holomorphic $1$-forms on $S$, and by $H_{1}(S, \mathbb{Z})$ the first homology group of $S.$ The {\it Jacobian variety} of $S$ is the ppav of dimension $g$ defined as$$JS=(\mathscr{H}^{1,0}(S, \mathbb{C}))^*/H_{1}(S, \mathbb{Z}),$$endowed with the principal polarization given by the geometric intersection number. 

The relevance of the Jacobian variety lies in the well-known Torelli's theorem, which asserts that two compact Riemann surfaces are isomorphic if and only if their Jacobians are isomorphic as ppavs.

Let $G$ be a finite group and let $W_1, \ldots, W_r$ be its rational irreducible representations. It is known that each action of $G$ on $S$ induces a $\mathbb{Q}$-algebra homomorphism  $\Phi : \mathbb{Q} [G]\to \text{End}_{\mathbb{Q}}(JS).$ Each $\alpha \in {\mathbb Q}[G]$ defines an abelian subvariety of $JS;$ namely, \begin{equation*} A_{\alpha} := {\rm Im} (\alpha)=\Phi (l\alpha)(JS) \subset JS,\end{equation*}where $l$ is some positive integer chosen such that $l\alpha \in \mbox{End}(JS)$.

The decomposition  $1 = e_1 + \cdots + e_r\in \mathbb{Q}[G]$, where $e_i$ is a central idempotent (uniquely determined and canonically computed from $W_i$) yields an isogeny $$JS \sim A_{e_1} \times \cdots \times A_{e_r}$$which is $G$-equivariant; this is called the {\it isotypical decomposition} of $JS.$ See \cite{l-r}.

Additionally, there are idempotents $f_{i1},\dots, f_{in_i} $ such that $e_i=f_{i1}+\dots +f_{in_i}$ where $n_i=d_{V_i}/s_{V_i},$ with $V_i$ being a complex irreducible representation associated to $W_i.$ These idempotents provide $n_i$ subvarieties of $JS$ which are isogenous to each other; let $B_i$ be one of them, for each $i$. Then
\begin{equation} \label{lg}
JS \sim_G B_{1}^{n_1} \times \cdots \times B_{r}^{n_r}
\end{equation}
which is called the {\it group algebra decomposition} of $JS$ with respect to $G.$ See \cite{cr}. 

If $W_1$ denotes the trivial representation, then $n_1=1$ and $B_{1} \sim J(S/G)$.

\s

Let $H$ be a subgroup of $G.$  It was proved in \cite{cr}  that the group algebra decomposition (\ref{lg}) of $JS$ with respect to $G$ yields the following isogeny decomposition:
\begin{equation} \label{lg2}
J(S/H) \sim B_{1}^{n_1^H} \times \cdots \times B_{r}^{n_r^H} \hspace{0.5 cm} \mbox{where} \,\,\,\, n_i^H=d_{V_i}^H/s_{V_i}
\end{equation}with $d_{V_i}^H$ denoting the dimension of the vector subspace $V_i^H$ of $V_i$ consisting of those elements which are fixed under $H.$

The previous result provides a criterion to identify if a factor in (\ref{lg}) is isogenous to the Jacobian of a quotient of $S$ (c.f. \cite{yo}). Namely, if a subgroup $N$ of $G$ satisfies $d_{V_i}^{N}=s_{V_i}$ for some fixed $2 \le i \le r$ and $d_{V_l}^{N} = 0$ for all $l \neq i$  
such that $B_{l} \neq 0,$ then  \begin{equation} \label{lg3}
B_{i} \sim J(S/N).
\end{equation}

Let us now suppose that $\tau =(\gamma; [m_1,C_1], \ldots, [m_l, C_l])$ is the geometric signature of the action of $G$ on $S.$ Let $G_k$ be a representative of the conjugacy class $C_k$ for $1 \le k \le l.$ In \cite{yoibero} it was proved that the dimension of the factor $B_{i}$ in (\ref{lg}) is
\begin{equation}\label{dimensiones}
\dim (B_{i})=k_i(d_{V_i}(\gamma -1)+\tfrac{1}{2} \Sigma_{k=1}^l (d_{V_i}-d_{V_i}^{G_k} ))
\end{equation}
where $k_i$ is the degree of the extension $\mathbb{Q} \le L_{V_i}.$ To avoid confusion, we shall write $\mbox{dim}_{\tau}(B_i)$ instead of $\mbox{dim}(B_i)$ to refer to the dependence on $\tau.$

For decompositions of Jacobians with respect to special groups, see, for example, the articles  \cite{d1}, \cite{CHQ}, \cite{nos}, \cite{IJR}, \cite{KR},   \cite{PA}, \cite{rubiyo}, \cite{rubiyo2} and \cite{d4}.

\section{Algebraic description of $\mathcal{F}_q$} \label{s3}
Let $q \ge 5$ be a prime number. Let $S$ denote a Riemann surface in the family $\mathcal{F}_q$ and 
let $$G=\langle r,s : r^{2q}=s^2=(sr)^2=1\rangle \cong \mathbf{D}_{2q}$$denote its full automorphism group. We recall that the quotient Riemann surface $\mathcal{}S/G$ has genus zero, and that the associated $4q$-fold branched regular covering map$$\pi_G: S \to S/G \cong \mathbb{P}^1$$ramifies over four values; three ramification values marked with 2 and one ramification value marked with $2q.$ We can assume the action to be represented by the generating vector $(s, sr^{-2}, r^q, r^{q+2}).$ In addition, up to a M\"{o}bius transformation, we can assume the branch values to be $\infty, 0, 1$ marked with 2 and $\lambda \in \mathbb{C}-\{0,1\}$ marked with $2q.$

As we will discuss later (see Remark \ref{conocidos}), $\lambda$ must be different from the exceptional values $-1, \tfrac{1}{2}, 2, \gamma, \gamma^2$ where $\gamma^3=-1.$ If we denote by $$\Omega:=\mathbb{C}-\{0, \pm 1, \tfrac{1}{2}, 2, \gamma, \gamma^2 \}$$the set of {\it admissible} parameters, then the family $\mathcal{F}_q$ can be understood by means of an everywhere maximal rank holomorphic map \begin{equation*} \label{azul} h : \mathcal{F}_q \to \Omega \end{equation*}in such a way that the fibers of $h$ agree with the Riemann surfaces in $\mathcal{F}_q.$ See, for example, Section 6.2 in \cite{Hubbard}.

\s

We denote by $S_{\lambda}$ the Riemann surface $h^{-1}(\lambda)$ and by $G_{\lambda} \cong G$ its full automorphism group.

\begin{theo} \label{modelo}
Let $\lambda \in \Omega.$  Then $S_{\lambda}$ is isomorphic to the Riemann surface defined by the normalization of the hyperelliptic algebraic curve$$y^2=x(x^{2q}+2 \tfrac{1+\lambda}{1-\lambda } x^{q} + 1).$$
\end{theo}

\begin{proof} Following \cite[Remark 9]{costa}, the Riemann surface $S_{\lambda}$ is hyperelliptic; the hyperelliptic involution being represented by $r^q.$ In other words, the Riemann surface $R_{\lambda}:=S_{\lambda}/\langle r^q \rangle$ has genus zero, and the associated two-fold branched regular covering map $$\pi_1: S_{\lambda} \to R_{\lambda}$$ramifies over $2q+2$ values; let us denote these values by $\alpha_1, \ldots, \alpha_{2q+2}.$ Let $$\pi_2: R_{\lambda} \to R_{\lambda}/K\cong S_{\lambda}/G \cong \mathbb{P}^1$$denote the  $2q$-fold branched regular covering map associated to the action of the quotient group $K=G/\langle r^q \rangle \cong \mathbf{D}_{q}$ on $R_{\lambda}.$ The following diagram commutes.  $$
\begin{tikzpicture}[node distance=1 cm, auto]
  \node (P) {$S_{\lambda}$};
  \node (A) [below of=P, left of=P] {$\mathbb{P}^1$};
  \node (C) [below of=B, right of=P] {$R_{\lambda}$};
  \draw[->] (P) to node [swap] {$\pi_G$} (A);
  \draw[->] (P) to node {$\pi_1$} (C);
  \draw[->] (C) to node [swap] {$\pi_2$} (A);
\end{tikzpicture}
$$

\s
{\bf Claim.} The ramification values of $\pi_2$ are: $\infty$ marked with two, $0$ marked with two, and $\lambda$ marked with $q.$ Moreover, among the ramification values of $\pi_1$ only two of them are ramification points of $\pi_2$, these points lying over $\lambda$ by $\pi_2.$

\s

We proceed to study carefully the ramification data associated to the coverings in the previous commutative diagram. To do that, we follow \cite[Section 3.1]{yoibero}.

\s
\begin{enumerate}
\item[(a)] The fiber over $\infty$ by $\pi_G$ consists of $2q$ different points, say $\beta'_1, \ldots, \beta'_{2q}.$ The stabilizer subgroup of $\beta'_j$ in $G$ is of the form $\langle sr^{-2t} \rangle$ for a suitable choice of $t.$ Now, as $|\langle sr^{-2t}  \rangle \cap \langle r^q \rangle|=1$ for each choice of $t$, it follows that $\beta'_j$ is not a branch point of $\pi_1.$ Thus, over $\infty$ by $\pi_2$ there are exactly $q$ different points, say $$\{\beta_1, \ldots, \beta_q\}=\pi_1(\{\beta'_1, \ldots, \beta'_{2q}\}),$$showing that $\infty$ is a ramification value of $\pi_2$ marked with two.
\s

\item[(b)] The fiber over $0$ by $\pi_G$ consists of $2q$ different points, say $\gamma'_1, \ldots, \gamma'_{2q}.$  As argued in (a), $\gamma'_j$ is not a branch point of $\pi_1$  and over $0$ by $\pi_2$ there are exactly $q$ different points, say $$\{\gamma_1, \ldots, \gamma_q\}=\pi_1(\{\gamma'_1, \ldots, \gamma'_{2q}\}),$$showing that $0$ is a ramification value of $\pi_2$ marked with two.

\s

\item[(c)] The fiber over $1$ by $\pi_G$ consists of $2q$ different points, say $\alpha'_1, \ldots, \alpha'_{2q}.$ The  stabilizer subgroup of $\alpha'_j$ in $G$ is $\langle r^q \rangle$ and therefore $\alpha'_j$ is a branch point marked with two of $\pi_1$ for each $j.$ Thus, over $1$ by $\pi_2$ there are exactly $2q$ different points, say $$\alpha_j=\pi_1(\alpha'_j), \hspace{0,3 cm} \mbox{for } j \in \{1, \ldots, 2q\}$$showing that $1$ is not ramification value of $\pi_2.$

\s

\item[(d)] The fiber over $\lambda$ by $\pi_G$ consists of 2 different points, say ${\alpha}'_{2q+1}$ and ${\alpha}'_{2q+2}.$ The stabilizer subgroup of ${\alpha}'_{2q+1}$ and ${\alpha}'_{2q+2}$ in $G$ is $\langle r \rangle.$ Now, as $|\langle r  \rangle \cap \langle r^q \rangle|=2,$ it follows that ${\alpha}'_{2q+1}$ and ${\alpha}'_{2q+2}$ are branch points of $\pi_1$ marked with $2$. Thus, over $\lambda$ by $\pi_2$ there are exactly two different points: $$\alpha_{2q+1} =\pi_1({\alpha}'_{2q+1}) \hspace{0.3 cm} \mbox{and}\hspace{0.3 cm} \alpha_{2q+2} =\pi_1({\alpha}'_{2q+2}),$$showing that $\lambda$ is a ramification value of $\pi_2$ marked with $q$.     
\end{enumerate}

The proof of the claim is done.
\s

Then, without loss of generality, we can suppose $K$ to be generated by  $$a(z)=\omega_{2q}^2z \hspace{0,3 cm} \mbox{and} \hspace{0,3 cm} b(z)=\tfrac{1}{z}$$where $\omega_t=\mbox{exp}(\tfrac{2 \pi i}{t}),$ and that:
\begin{enumerate}
\item[(a)] the $q$ branch points of $\pi_2$ over $\infty$ are $\beta_j=\omega_{2q}^{2j-1}$ for $1 \le j \le q.$
\vspace{0,1 cm}
\item[(b)] the $q$ branch points of $\pi_2$ over $0$ are $\gamma_j=\omega_{2q}^{2j}$ for $0 \le j < q$.
\vspace{0,2 cm}
\item[(c)] the two branch points of $\pi_2$ over $\lambda$ are $\alpha_{2q+1}=0$ and $\alpha_{2q+2}=\infty.$
\end{enumerate}
\vspace{0,1 cm}

It follows that $S_{\lambda}$ is isomorphic to the Riemann surface defined by the normalization of the hyperelliptic algebraic curve$$y^2=x(x-\alpha_1) \cdots (x-\alpha_{2q}),$$and it only remains to prove that $\alpha_1, \ldots, \alpha_{2q}$ are the $2q$ different solutions of the polynomial equation $$z^{2q}+ 2\tfrac{1+\lambda}{1-\lambda } z^{q} + 1=0.$$

Now, by virtue of the claim, to accomplish this task we need to exhibit a $2q$-fold branched regular covering map  $\Pi : \mathbb{P}^1 \to \mathbb{P}^1$ admitting $\langle a, b \rangle \cong \mathbf{D}_q$ as its deck group, such that
$\Pi(\infty)=\Pi(0)=\lambda$ and \begin{displaymath}
\Pi(\omega_{2q}^k)= \left\{ \begin{array}{ll}
\infty & \mbox{if } k \mbox{ is odd; } \\
0 & \mbox{if } k \mbox{ is even.}
\end{array} \right. \end{displaymath}

It is straightforward  to check that $\Pi(z) = \lambda \cdot \tfrac{z^{2q}-2z^q+1}{z^{2q}+2z^q+1}$ is the desired map, and the proof follows directly after noticing that $$\{\alpha_1, \ldots, \alpha_{2q}\}=\Pi^{-1}(1).$$\end{proof}

\begin{rema} As we shall see later (see Theorem \ref{lreal}) among the members of the family $\mathcal{F}_q$ there are some of them which admit anticonformal involutions. In this case, the previous result can also be obtained as a consequence of the results of \cite{gamboa}.
\end{rema}

\begin{theo} 
Let $\lambda \in \Omega.$ In the algebraic model of Theorem \ref{modelo} the full automorphism group of $S_{\lambda}$ is generated by the transformations $$r(x,y)=(\omega_q x, \omega_{2q} y) \hspace{0,3 cm} \mbox{and} \hspace{0,3 cm} s(x,y)=(\tfrac{1}{x}, \tfrac{y}{x^{q+1}})$$where $\omega_t=\mbox{exp}(\tfrac{2 \pi i}{t}).$  
\end{theo}

\begin{proof} The fact that the transformations $r$ and $s$ are indeed automorphisms of $S_{\lambda}$ follows from an easy computation. Note that $s$ has order two, $r$ has order $2q$ and $$sr(x,y)=(\tfrac{1}{\omega_q                                 x}, \tfrac{y}{\omega_{2q}x^{q+1}})$$has order two; thus, $r$ and $s$ generate a group of order $4q$ isomorphic to $\mathbf{D}_{2q}.$

The proof of the theorem follows after noticing that the stabilizer subgroup of each ramification point of the regular covering map associated to the action of $\langle r, s \rangle$ is conjugate to the group generated by either $s, sr^{-2}, r^q$ or $r^{q+2}.$  In fact:

\begin{enumerate}
\item[(a)] each power of $r^{q+2}$ has two fixed points with stabilizer subgroup $\langle r \rangle,$ 
\item[(b)] $r^q$ has $2q$ fixed points with stabilizer subgroup $\langle r^q \rangle,$ 
\item[(c)] if $t$ is odd then the involution $sr^t$ does not have fixed points, and
\item[(d)] if $t$ is even then the involution $sr^t$ has four fixed points; the $G$-orbit of each of them has cardinality $2q,$ and the stabilizer subgroup of each point in this orbit is conjugate to $\langle s \rangle.$
\end{enumerate}
The proof is done.
\end{proof}

We anticipate the fact that the Jacobian variety $JS_{\lambda}$ decomposes, up to isogeny, as a product of an elliptic curve $E_{\lambda}$ and (two copies of) the Jacobian of a Riemann surface of genus $\tfrac{q-1}{2}$. 

The next proposition describes algebraically the elliptic curve $E_{\lambda}.$ 

\begin{prop}\label{eli} Let $\lambda \in \Omega,$ and consider the following subgroup of $G$ $$H_4=\langle r^{-2}, sr^{-1}\rangle \cong \mathbf{D}_q.$$  Then the quotient Riemann surface $E_{\lambda}$ given by the action of $H_4$ on $S_{\lambda}$ has genus one, and it is endowed with a two-fold regular covering map over the projective line which ramifies over $\infty, 0, 1$ and $\lambda.$ In particular, $E_{\lambda}$ is isomorphic to the Riemann surface defined by the elliptic curve $$y^2=x(x-1)(x-\lambda).$$
\end{prop}

\begin{proof} The normality of $H_4$ as a subgroup of $G$ implies that the quotient group $H:=G/H_4 \cong \mathbb{Z}_2$ acts on $E_{\lambda}.$ Let  $$\pi_1: S_{\lambda} \to E_{\lambda} \hspace{0,3 cm} \mbox{and} \hspace{0,3 cm} \pi_2: E_{\lambda} \to \mathbb{P}^1$$denote the branched regular covering maps given by the action of $H_4$ on $S_{\lambda},$ and by the action of $H$ on $E_{\lambda}$ respectively. The following diagram commutes:$$
\begin{tikzpicture}[node distance=1 cm, auto]
  \node (P) {$S_{\lambda}$};
  \node (A) [below of=P, left of=P] {$\mathbb{P}^1$};
  \node (C) [below of=B, right of=P] {$E_{\lambda}$};
  \draw[->] (P) to node [swap] {$\pi_G$} (A);
  \draw[->] (P) to node {$\pi_1$} (C);
  \draw[->] (C) to node [swap] {$\pi_2$} (A);
\end{tikzpicture}
$$

Following the same notations used in the proof of Theorem \ref{modelo}, we can assert that: 
\begin{enumerate}
\item[(a)]  the $2q$ different points $\beta'_1, \ldots, \beta'_{2q}$ ($\gamma'_1, \ldots, \gamma'_{2q}$ and $\alpha'_1, \ldots, \alpha'_{2q},$ respectively) lying over $\infty$ (over 0 and over 1,  respectively) by $\pi_G$ are not ramification points of $\pi_1$ and therefore they are sent to one point in $E_{\lambda}.$ Thereby, $\infty$ (0 and 1, respectively) is a branch value of $\pi_2$ marked with two,

\s
%
%\item[(b)]  the $2q$ different points $\gamma'_1, \ldots, \gamma'_{2q}$ lying over $0$ by $\pi_G$ are not ramification points of $\pi_1$ and therefore they are sent to one point in $E_{\lambda}.$ Thereby, $0$ is a branch value of $\pi_2$ marked with two,
%
%\s
%
%\item[(c)]  the $2q$ different points $\alpha'_1, \ldots, \alpha'_{2q}$ lying over $1$ by $\pi_G$ are not ramification points of $\pi_1$ and therefore they are sent to one point in $E_{\lambda}.$ Thereby, $1$ is a branch value of $\pi_2$ marked with two, and
%
%\s

\item[(b)]  the two different points $\alpha'_{2q+1}$ and $ \alpha'_{2q+2}$ lying over $\lambda$ by $\pi_G$ are  ramification points of $\pi_1;$ the intersection of their stabilizer subgroup with $H_4$ having order $q.$ Thus, they are sent to one point in $E_{\lambda},$ and $\lambda$ is a branch value of $\pi_2$ marked with two. 
\end{enumerate} 

Thus, $E_{\lambda}$ is endowed with a two-fold regular covering map over the projective line, with four branch values. As the genus of $E_{\lambda}$ is one (Riemann-Hurwitz formula), the result follows.
\end{proof}

\section{The group algebra decomposition of $JS$} \label{s4}

In this section we consider the Jacobian variety $JS_{\lambda}$  and study the group algebra decomposition of it with respect to its full automorphism group $G.$ In order to state the results, we start by studying the representations of $G$ and the generating vectors representing the action of $G$ on $S_{\lambda}.$

It is well-known that the dihedral group $$G=\langle r,s : r^{2q}=s^2=(sr)^2=1\rangle$$has, up to equivalence, 4 complex irreducible representations of degree one; namely,\begin{displaymath}
V_1 : \left\{ \begin{array}{ll}
r \to 1 \\
s \to 1 
\end{array} \right. V_2 : \left\{ \begin{array}{ll}
r \to 1 \\
s \to -1 
\end{array} \right. V_3 : \left\{ \begin{array}{ll}
r \to -1 \\
s \to 1 
\end{array} \right. V_4 : \left\{ \begin{array}{ll}
r \to -1 \\
s \to -1 
\end{array} \right.
\end{displaymath}and $q-1$ complex irreducible representations of degree two; namely,
\begin{displaymath}
V_{k+4} : \, r \mapsto \mbox{diag}(\omega_{2q}^k, \bar{\omega}_{2q}^k) 
 ,\  s \mapsto \left( \begin{smallmatrix}
0 & 1\\
1 & 0 \\
\end{smallmatrix} \right)
\end{displaymath}for $1 \le k \le q-1$ and $\omega_t=\mbox{exp}(\tfrac{2 \pi i}{t}).$

\begin{lemm} \mbox{}
\begin{enumerate}
\item[(1)] The rational irreducible representations of $G,$ up to equivalence, are:
\begin{enumerate}
\item[(a)] four of degree 1; namely $W_i:=V_i$ for $1 \le i \le 4$ and \item[(b)] two of degree $q-1$; namely $$W_5=\oplus_{\sigma \in G_5} V_5^{\sigma} \hspace{0.3 cm} \mbox{and} \hspace{0.3 cm} W_6=\oplus_{\sigma \in G_6}  V_6^{\sigma}$$where $G_5$ and $G_6$ denote the Galois group associated to the extensions $\mathbb{Q} \le \mathbb{Q}(\omega_{2q} + \bar{\omega}_{2q})$ and $\mathbb{Q} \le \mathbb{Q}(\omega_{q} + \bar{\omega}_{q})$ respectively, and $\omega_t=\mbox{exp}(\tfrac{2 \pi i}{t}).$
\end{enumerate}
\s
\item[(2)] Let $\lambda \in \Omega.$ The group algebra decomposition of $JS_{\lambda}$ with respect to $G$ is $$JS_{\lambda} \sim_G B_1 \times B_2 \times B_3 \times B_4 \times B_5^2 \times B_6^2$$where $B_{j}$ stands for the factor associated to the representation $W_{j}.$
\end{enumerate}
\end{lemm}

\begin{proof}
The proof of part (1) follows directly from the way in which the rational irreducible representations of a group are constructed (see  Subsection \ref{rep}).

The proof of part (2) is a direct consequence of (1) together with the group algebra decomposition theorem (see  Subsection \ref{GAD}).
\end{proof}

As anticipated in the introduction of this article, to compute the dimension of the factors $B_j$ (which may be zero) we need to choose a generating vector representing the action of $G$ on $S_{\lambda};$ the following lemma provides all those possible choices.

\begin{lemm} \label{vgopciones} Let $\sigma$ be a generating vector of $G$ of type $(2,2,2,2q)$. Then there exist integers $e_1, e_2$ with $e_1-e_2$ even and not congruent to $0$ modulo $2q$, such that
\begin{equation*} \label{vgpar} \sigma = (sr^{e_1}, sr^{e_2}, r^{q}, r^{e_1-e_2+q}) \end{equation*}
up to the action of the symmetric group  $\mathbf{S}_3$ over the first three entries.
\end{lemm}

\begin{proof}
Let us suppose that $$\sigma = (g_1, g_2, g_3, g_4=(g_1g_2g_3)^{-1})$$is a generating vector of $G$ of type $(2,2,2,2q).$ It is not difficult to see that $G$ has exactly three conjugacy classes of elements of order two; namely $$C_1=\{sr^{n} : 0 \le n < 2q \mbox{ even}\},  \hspace{0,3 cm}C_2=\{sr^{m} : 1 \le m < 2q \mbox{ odd}\}$$and $C_3=\{r^q\}.$ Moreover, there are $\tfrac{q-1}{2}$ conjugacy classes of elements of $G$ of order $2q$; namely $\{r^t, r^{-t} \}$ for each odd integer $1 \le t < q.$

Since $g_1$, $g_2$ and $g_3$ must generate $G$ and since their product $g_4^{-1}$ must have order $2q$, it is straightforward to see that:
\begin{enumerate}
\item[(a)] the elements $g_1, g_2$ and $g_3$ cannot belong simultaneously to only one of the conjugacy classes $C_1, C_2$ or $C_3,$
\item[(b)] the elements $g_1, g_2$ and $g_3$ cannot belong to three different  conjugacy classes $C_1, C_2$ and $C_3,$ and
\item[(c)] one (and only one) of the elements $g_1, g_2$ or $g_3$ must belong to the conjugacy class $C_3.$
\end{enumerate}

Hence, up to a permutation in $\mathbf{S}_3$, we may assume $\sigma$ to be of the form $$(sr^{e_1}, sr^{e_2}, r^{q}, r^{e_1-e_2+q})$$where $0 \le e_1,e_2 < 2q$ are simultaneously odd or simultaneously even. As $e_1-e_2+q$ must be coprime with $2q$, the difference $e_1-e_2$ is not congruent with zero modulo $2q$. The proof is done.
\end{proof}

\begin{rema} \mbox{}
\begin{enumerate}
\item[(a)] We should mention that the proof of the previous lemma could be derived from the proof of Theorem 7 in \cite{costa}. 
\item[(b)] Following \cite[Remark 8]{costa} there is a unique topological class of action of $\mathbf{D}_{2q}$ on Riemann surfaces of genus $q$ with signature $(2,2,2,2q);$ consequently, every generating vector of $G$ of the desired type can be chosen to represent the action of $G$ on $S_{\lambda}.$
\end{enumerate}
\end{rema}

We now proceed to analyze how the choice of the generating vector changes the dimension of the factors arising in the group algebra decomposition of $JS_{\lambda}$ with respect to $G.$ To accomplish this task, it is convenient to bring in the following equivalence relation:

\begin{defi} Two generating vectors $\sigma_1$ and $\sigma_2$ are termed {\it essentially equal} with respect to the action of $G$ on $S_{\lambda}$ if $\dim_{\tau_1}(B_j) = \dim_{\tau_2}(B_j)$ 
for all $j,$ where $\tau_i$ is the geometric signature associated to $\sigma_i.$  
\end{defi}

\begin{lemm} \label{todos} Each generating vector of $G$ of type $(2,2,2,2q)$ is essentially equal to  $$\sigma_{0}=(s, sr^{-2}, r^q, r^{q+2}) \hspace{0.4 cm} \mbox{ or to } \hspace{0.4 cm} \sigma_{1}=(sr, sr^{-1}, r^q, r^{q+2}).$$
\end{lemm}

\begin{proof} Let $\sigma$ be a generating vector of $G$ of type $(2,2,2,2q)$ for the action of $G$ on $S_{\lambda}.$ By Lemma \ref{vgopciones} we can suppose $$\sigma = (sr^{e_1}, sr^{e_2}, r^{q}, r^{e_1-e_2+q})$$for some integers $e_1, e_2$ with $e_1-e_2$ even and not congruent to $0$ modulo $2q,$ up to the action of  $\mathbf{S}_3$ on the first three entries. The action of $\iota \in \mathbf{S}_3$ over the three first entries produces the following change on the fourth one:  $$r^{e_1-e_2+q} \mapsto \iota(r^{e_1-e_2+q})=r^{\pm(e_1-e_2+q)}$$sending $r^{e_1-e_2+q}$ into either itself or its inverse. Hence, in spite of the fact that the corresponding geometric signature changes under permutations in $\mathbf{S}_3,$ the dimension of the each factor $B_j$ remains the same ($\iota$ permutes the summands in the sum \eqref{dimensiones}); thus the generating vectors $\sigma$ and $\iota(\sigma)$ are essentially equal.

We remark the obvious observation that the geometric signature associated to a given generating vector is kept invariant under inner automorphisms of the group. Now, after conjugating every element in $\sigma$ by \begin{displaymath}
 \left\{ \begin{array}{ll}
r^{-e_1/2} & \mbox{if } e_1 \mbox{ and } e_2 \mbox{ are even; }\\
r^{-(e_1+1)/2} & \mbox{if } e_1 \mbox{ and } e_2 \mbox{ are odd,}
\end{array} \right. \end{displaymath}
we obtain normalized generating vectors$$\sigma_{0, n}:=(s, sr^{-n}, r^q, r^{q+n}) \hspace{0,3 cm} \mbox{and} \hspace{0,3 cm} \sigma_{1, n}:=(sr, sr^{1-n}, r^q, r^{q+n})$$ for $e_1, e_2$ even, and for $e_1, e_2$ odd respectively, where $n=e_1-e_2.$

Note that if $n$ and $m$ are distinct even numbers, then $\sigma_{0, n}$ and $\sigma_{0, m}$ are essentially equal, and $\sigma_{1, n}$ and $\sigma_{1, m}$ are essentially equal. 

Hence, the result follows after verifying that  $\sigma_{0} = \sigma_{0, 2}$ and $\sigma_{1} =\sigma_{1, 2}$ are not essentially equal; this follows from $\dim_{\tau_{0}}(B_3)=0$ and $\dim_{\tau_{1}}(B_3)=1. $ \end{proof}

\begin{prop} \label{theo1}
Let $\lambda \in \Omega,$ and consider the group algebra decomposition of $JS_{\lambda}$ with respect to $G$ $$JS_{\lambda} \sim_G B_1 \times B_2 \times B_3 \times B_4 \times B_5^2 \times B_6^2.$$

If $\tau_0$ denotes the geometric signature associated to $\sigma_{0}=(s, sr^{-2}, r^q, r^{q+2})$ then\begin{displaymath}
\mbox{dim}_{\tau_{0}} (B_{j})= \left\{ \begin{array}{ll}
 \,\,\,\, 0& \textrm{if $j = 0,1,2,3,6$}\\
  \,\,\,\,  1& \textrm{if $j = 4$}\\
\tfrac{q-1}{2}  & \textrm{if $j=5$}
  \end{array} \right.
\end{displaymath}

If $\tau_1$ denotes the geometric signature associated to $\sigma_{1}=(sr, sr^{-1}, r^q, r^{q+2}),$ then 
\begin{displaymath}
\mbox{dim}_{\tau_{0}} (B_{j})= \left\{ \begin{array}{ll}
 \,\,\,\, 0& \textrm{if $j = 0,1,2,4,6$}\\
  \,\,\,\,  1& \textrm{if $j = 3$}\\
\tfrac{q-1}{2}  & \textrm{if $j=5$}
  \end{array} \right.
\end{displaymath} In particular, $JS_{\lambda}$ contains an elliptic curve.
\end{prop}

\begin{proof} The genus of the quotient $S/G$ is zero; thus, $\mbox{dim}_{\tau_0}(B_1)=\mbox{dim}_{\tau_1}(B_1)=0.$ The table below summarized the dimension of the vector subspaces of each $V_j$ fixed under the cyclic subgroup $\langle g \rangle,$ for each $g$ arising in the signatures $\sigma_{0}$ and $\sigma_{1}.$\begin{center}
\scalebox{0.7}{
\begin{tabular}{|c|c|c|c|c|c|c|}  \hline
\, &  $\langle s \rangle$ & $\langle sr \rangle$ &   $\langle sr^{-2} \rangle$ & $\langle sr^{-1} \rangle$ &  $\langle r^q \rangle$  & $\langle r^{q+2} \rangle$  \\ \hline
$V_{2}$  & 0 & 0 & 0 & 0 & 1 & 1 \\ \hline
$V_{3}$  & 1 & 0 & 1 & 0 & 0 & 0\\ \hline
$V_{4}$  & 0 & 1 & 0 & 1 & 0 & 0 \\ \hline
$V_{5}$  & 1 & 1 & 1 & 1 & 0 & 0 \\ \hline
$V_{6}$  & 1 & 1 & 1 & 1 & 2 & 0\\ \hline
\end{tabular}}\end{center}Now, the result follows directly as an application of \eqref{dimensiones}.

\end{proof}

\begin{theo} 
Let $\lambda \in \Omega$. The group algebra decomposition of $JS_{\lambda}$ with respect to $G$ does not depend on the choice of the generating vector. 
\end{theo}

\begin{proof}
By Lemmata \ref{vgopciones} and \ref{todos}, we only need to compare the decompositions associated to $\sigma_0$ and $\sigma_1.$ By Proposition \ref{theo1}, these decompositions are $$JS_{\lambda} \sim_{G, \sigma_0} B_4 \times B_5^2 \hspace{0.3 cm} \mbox{and} \hspace{0.3 cm}JS_{\lambda} \sim_{G, \sigma_1} B_3 \times B_5^2 $$respectively, showing that $B_3$ and $B_4$ are isogenous. We claim that, in addition, $B_4$ and $B_5$ are equal. Indeed, note that the generating vectors $\sigma_{0}$ and $\sigma_{1}$ are identified by the action of the outer automorphism $\Phi$ of $G$ defined by $r \mapsto r, s \mapsto sr.$ 

Accordingly, at the level of rational irreducible representations, $W_3$ is sent by $\Phi$ to $W_4.$ As a matter of fact, this shows that the roles played by $B_3$ and $B_4$ are interchanged according to the choice of the generating vector employed to compute the dimensions.

The proof is done.
\end{proof}

\begin{rema}The independence of the group algebra decomposition on the choice of the generating vector when there is a unique topological class of action is not new and was firstly noticed by Rojas in \cite[Example 4.3]{yoibero} when she considered the Weyl group $\mathbb{Z}_2^3 \rtimes \mathbf{S}_3$ acting on a Riemann surface of genus three with signature $(2,4,6).$ 

Very recently, the same phenomenon has been noticed by Izquierdo, Jim\'enez and Rojas itself in \cite{IJR} when they studied a two-dimensional family of  Riemann surfaces of genus $2n-1$ with action of $\mathbf{D}_{2n}$ with signature $(2,2,2,2,n)$. 

It is worth recalling that in the two aforementioned cases as well as in the case of the family $\mathcal{F}_q$, the existence of outer automorphisms of the group is the key ingredient. Based on the evidence of explicit examples, it seems reasonable to ask if this is the general situation; however, according to the knowledge of the author, it has not been proved a general result on this respect.
\end{rema}

From now on, we assume the action of $G$ on $S_{\lambda}$ to be determined by the generating vector $\sigma_{0}$ and, in consequence, the group algebra decomposition of $JS_{\lambda}$ with respect to $G$ to be of the form $$JS_{\lambda} \sim_G B_4 \times B_5^2.$$

The following result shows that the factors $B_4$ and $B_5$ have a geometric meaning.

\begin{theo} \label{theo2} Let $\lambda \in \Omega.$ Consider the subgroups $H_4=\langle r^{-2}, sr^{-1} \rangle$ and $H_5=\langle s \rangle$ of $G,$ and the  quotient Riemann surfaces $E_{\lambda}$ and $C_{\lambda}$ given by the action of $H_4$ and of $H_5$ on $S_{\lambda},$ respectively. Then  $$B_4 \sim JE_{\lambda} \hspace{0.3 cm} 
\mbox{and} \hspace{0.3 cm}  B_5 \sim JC_{\lambda} .$$

In particular, $JS_{\lambda}$ decomposes into  a product of Jacobians as follows: \begin{equation*} \label{jacos}JS_{\lambda} \sim_G JE_{\lambda} \times JC_{\lambda}^2.\end{equation*}
 \end{theo}

\begin{proof}
The dimension of the complex vector subspaces of $V_4$ and $V_5$ fixed under the subgroups $H_4$ and $H_5$ are$$d_{V_4}^{H_4}=d_{V_5}^{H_5}=1 \hspace{0.3 cm} \mbox{and} \hspace{0.3 cm} d_{V_4}^{H_5}=d_{V_5}^{H_4}=0.$$

Thus, the result follows after applying the criterion to identify factors in the group algebra decomposition of $JS_{\lambda}$ as Jacobians of quotients of $S_{\lambda}$ (as explained in Subsection \ref{GAD}; see equations \eqref{lg2} and \eqref{lg3}) together with Proposition \ref{theo1}. \end{proof}

\begin{rema} Note that $C_{\lambda}$ is an irregular $2q$-gonal Riemann surface of genus $\tfrac{q-1}{2}.$ An explicit algebraic description of $E_{\lambda}$ has been obtained in Proposition \ref{eli}.\end{rema}

\section{Fields of definition}\label{s5} Let $\mbox{Gal}(\mathbb{C})$ denote the group of field automorphisms of $\mathbb{C}.$ Let $X \subset \mathbb{P}^n$ be a (smooth algebraic) variety and $\sigma \in \mbox{Gal}(\mathbb{C})$. We shall denote by $X^{\sigma}$ the variety defined by the polynomials obtained after applying $\sigma$ to the coefficients of the polynomials which define $X.$

\s

Let $k$ be a subfield of $\mathbb{C}$ and let $\mbox{Gal}(\mathbb{C}/k)$ be the subgroup of $\mbox{Gal}(\mathbb{C})$ consisting of those automorphisms which fix the elements in $k.$ We shall say that $X$ {\it is defined over $k$} if $X=X^{\sigma}$ for all $\sigma \in \mbox{Gal}(\mathbb{C}/k).$ We shall say that $X$ {\it can be defined over $k$} (or that {\it $k$ is a field of definition for $X$}) if there exists an isomorphism $X \to Y$ into a variety $Y \subset \mathbb{P}^m$ which is defined over $k.$ 

\s

By considering the explicit algebraic description of $S_{\lambda}$ provided in Theorem \ref{modelo}, in this section we derive results concerning the field of definitions of $S_{\lambda}$ according to the value of $\lambda.$
 
\subsection{Real Riemann surfaces} An algebraic variety is called {\it real} if it can be defined over the field of the real numbers; equivalently, if it admits an anticonformal involution (i.e. an anticonformal automorphism of order two).

Following \cite[Section 6]{costa}, when we consider the family $\mathcal{F}_q$ as a complex subvariety of the moduli space $\mathscr{M}_q$ of compact Riemann surfaces of genus $q,$ it is isomorphic to the projective line minus three points. Furthermore, $\mathcal{F}_q \subset \mathscr{M}_q$ admits an anticonformal involution whose fixed point set consists of  points representing real Riemann surfaces.

Real Riemann surfaces have been extensively studied; see, for example, \cite{libroemilio}.
 The following result shows that among the Riemann surfaces $S_{\lambda}$ in $\mathcal{F}_q,$ the real ones can be easily recognized according to the value of $\lambda.$ More precisely:

\begin{theo} \label{lreal} Let $\lambda \in \Omega$. Then the following statements are equivalent:
\begin{enumerate}
\item[(a)] $S_{\lambda}$ is a real Riemann surface.
\item[(b)] $JS_{\lambda}$ is a real algebraic variety.
\item[(c)] $\lambda \in \{ \bar{\lambda}, 1-\bar{\lambda}, 1/\bar{\lambda}, \bar{\lambda}/(1-\bar{\lambda})\}$
\end{enumerate}
\end{theo}

\begin{proof} The equivalence between the first two statements is well-known; indeed, following \cite[Theorem 1.1]{Milne}, a Riemann surface and its Jacobian variety can be defined over the same fields.

We now proceed to prove the equivalence between the statements (a) and (c). 

Let us assume that $S_{\lambda}$ is a real Riemann surface or, equivalently, that $S_{\lambda}$ admits an anticonformal involution, denoted by $f_{\lambda}:S_{\lambda} \to S_{\lambda}.$  It is clear that $f_{\lambda}Gf_{\lambda}^{-1}=G$ and therefore $f_{\lambda}$ gives rise to an anticonformal involution $g_{\lambda}: \mathcal{O}_{\lambda} \to  \mathcal{O}_{\lambda},$ where $\mathcal{O}_{\lambda}$ denotes the Riemann orbifold given by the action of $G$ on $S_{\lambda}.$

We recall that $\mathcal{O}_{\lambda}$ has genus zero and four marked points: $0, 1$ and $\infty$ marked with $2,$ and $\lambda$ marked with $2q.$ It follows that $g_{\lambda}$ is an extended M\"{o}bius transformation, i.e. $g_{\lambda}(z)=(a \bar{z}+b)/(c\bar{z}+d)$ with $a,b,c,d \in \mathbb{C}$ and $ad-bc \neq 0,$ satisfying $$g_{\lambda}(\lambda)=\lambda \hspace{0,3 cm} \mbox{and}  \hspace{0,3 cm} g_{\lambda}(\{\infty, 0, 1\})=\{\infty, 0, 1\}.$$

We only have four possibilities:
\begin{enumerate}
\item $g_{\lambda}$ fixes $\infty$ and permutes $0$ and $1.$ In this case $g_{\lambda}(z)=1-\bar{z}$ showing that $\lambda=1-\bar{\lambda}.$ 
\item $g_{\lambda}$ fixes $0$ and permutes $1$ and $\infty.$ In this case $g_{\lambda}(z)=\tfrac{\bar{z}}{1-\bar{z}}$ showing that $\lambda=\bar{\lambda}/(1-\bar{\lambda}).$ 
\item $g_{\lambda}$ fixes $1$ and permutes $\infty$ and $0.$ In this case $g_{\lambda}(z)=\tfrac{1}{\bar{z}}$ showing that $\lambda=1/\bar{\lambda}.$ 
\item $g_{\lambda}$ fixes $\infty, 0$ and $1.$ In this case $g_{\lambda}(z)=\bar{z}$ showing that $\lambda=\bar{\lambda}.$ 
\end{enumerate}Hence, $\lambda$ is as in statement $(c).$

Conversely, if $\lambda$ is as in statement (c), to construct explicitly an anticonformal involution is an easy task, and the proof is done.\end{proof}

\begin{rema} \label{conocidos} Following \cite[Theorem 14]{costa}, the real Riemann surfaces in the family $\mathcal{F}_q$ form three (one-real-dimensional) arcs. In addition, in order to compactify the union of these arcs in the Deligne-Mumford compactification of $\mathscr{M}_g,$ it was proved that it is enough to add to $\mathcal{F}_q$ three points: these points representing two nodal Riemann surfaces, and the {\it Wiman surface} of type II (this is the unique compact Riemann surface of genus $q$ admitting an automorphism of order $4q$; see \cite{Wi}). 

\s

The aforementioned results were obtained by using Teichm\"{u}ller theory and Fuchsian groups, among other techniques. Here, by considering the algebraic description of the Riemann surfaces in $\mathcal{F}_q$ given in Theorem \ref{modelo}, we are able to recover partly these results in a very explicit way as follows.

\s

The Riemann surfaces $S_{\lambda_1}$ and $S_{\lambda_2}$ are isomorphic if and only if $\lambda_2=T(\lambda_1)$ for some \begin{equation} \label{sim}T \in \mathbb{G}= \langle z \mapsto \tfrac{1}{z}, z \mapsto \tfrac{1}{1-z} \rangle \cong \mathbf{S}_3.\end{equation}  Observe that for the exceptional values $-1, \tfrac{1}{2}, 2, \gamma$ and $\gamma^2$ where $\gamma^{3}=-1,$ the Riemann surface $S_{\lambda}$ has more than $4q$ automorphisms.

Thus, the family $\mathcal{F}_q$ is isomorphic to the quotient of the parameter space $$\Omega=\mathbb{C}-\{0, \pm 1, \tfrac{1}{2}, 2, \gamma, \gamma^2 \}$$up to the action of $\mathbb{G}.$ Namely: $\Omega \to \Omega/\mathbb{G} \cong \mathcal{F}_q \cong \mathbb{C}-\{0,1\}.$

According to Theorem \ref{lreal}, the complex numbers $\lambda \in \Omega$ representing Riemann surfaces $S_{\lambda}$ which are real can be represented in the diagram below; the colored red points represent Riemann surfaces with more than $4q$ automorphisms (and therefore they do not belong to $\mathcal{F}_q$).
\begin{figure}[htp]
\centering
\includegraphics[width=4cm]{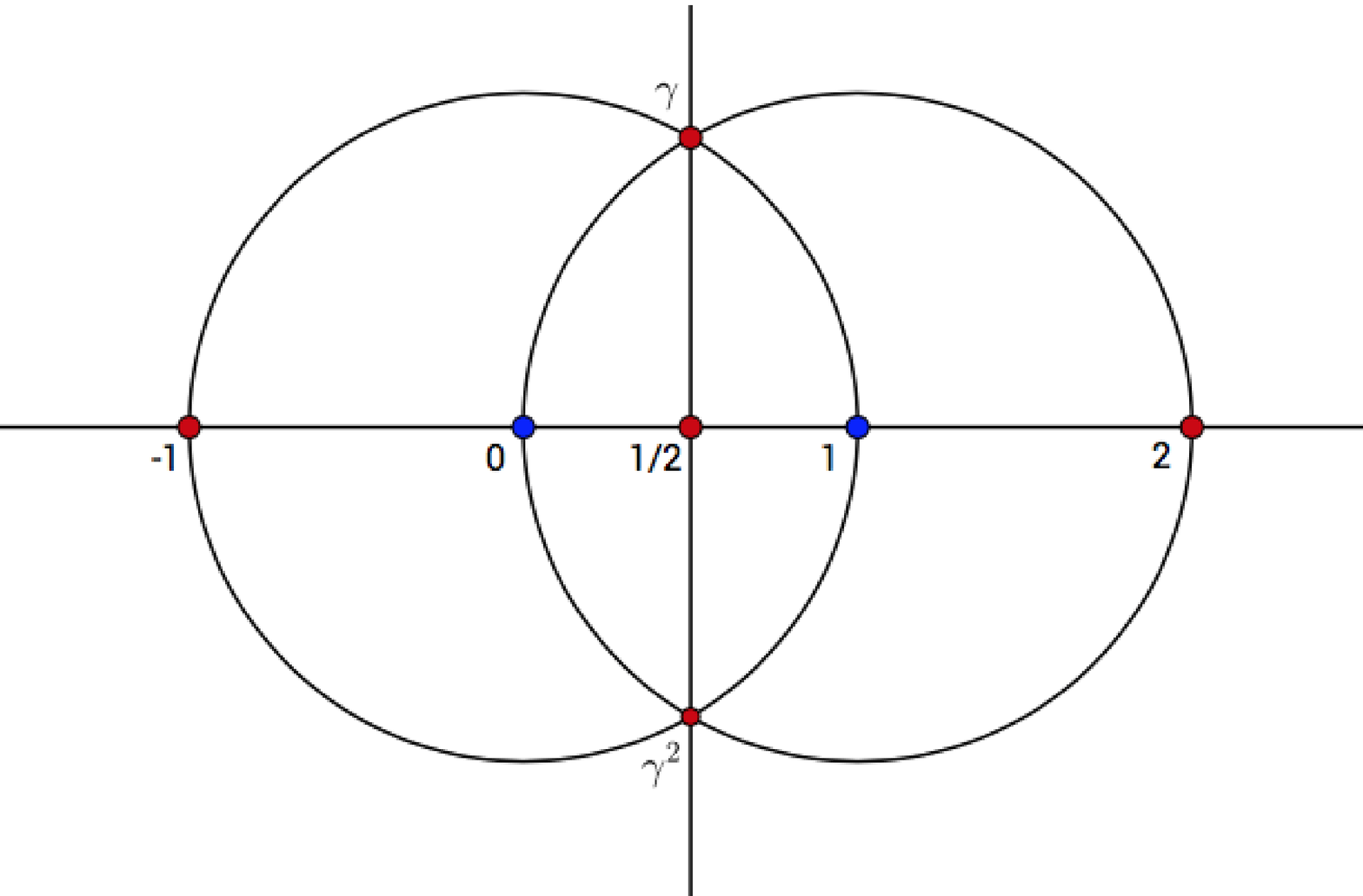}
\end{figure}

Note that a fundamental region for the action of $\mathbb{G}$ on $\Omega$ is given by $$\{ z \in \mathbb{C} : |z| < 1, \mbox{Re}(z) < \tfrac{1}{2}\}$$and, consequently, the subsets of $\mathcal{F}_q$ given by$$\Pi(\{ e^{i \theta} : \pi < \theta < \tfrac{\pi}{2}\}), \, \, \Pi(\{ z : |z-1|=1, |z| < 1\}) \mbox{ and } \Pi(]-1,0[)$$are the three arcs in $\mathcal{F}_q$ (denoted in \cite{costa} by $a_2, a_1$ and $b$ respectively) corresponding to real Riemann surfaces in $\mathcal{F}_q.$

Note that the limit point of $\mathcal{F}_q$ which connects the arcs $a_2$ and $b$ correspond to $S_{-1}$ and therefore, by Theorem \ref{modelo}, can be algebraically described by  $$y^2=x(x^{2q}+1).$$

The map $(x,y) \mapsto (-\omega_{4q} x, \omega_{8q} y )$ where $\omega_t = \mbox{exp}(\tfrac{2 \pi i}{t}),$ induces an isomorphism between $S_{-1}$ and the curve $$y^2=x(x^{2q}-1);$$this is the Wiman surface of type II. 
\end{rema}
%
%{\color{red} me gustar\'ia tener una matriz de periodo entera para esta superficie}
%
%%{\color{red} $S_0$ est\'a dada por $y^2=x(x^q+1)^2$ y $S_{\gamma}$ por $y^2=x(x^{2q}+2i \sqrt{3}x^q+1)$ }
%

\begin{rema}
After proving that an algebraic variety is real, to find explicit defining equations with real coefficients is, in general, a difficult problem. If $\lambda$ is real then a model for $S_{\lambda}$ in terms of equations with real coefficients is provided by Theorem \ref{modelo}. In the remaining cases, the construction of real equations can be done by applying the results of \cite{rubenyo}. 
\end{rema}

\subsection{Arithmetic Riemann surfaces} An algebraic variety is called {\it arithmetic} if it can be defined over a number field or, equivalently, over the algebraic closure $\overline{\mathbb{Q}}$ of the field of the rational numbers. 

A well-known result due to Belyi  ensures that a Riemann surface is arithmetic if and only if it admits a non-constant meromorphic function with three critical values; see \cite{Belyi}. For arithmetic complex surfaces an analogous result to Belyi's theorem was proved by
Gonz\'alez-Diez in \cite{lef} by considering the so-called Lefschetz maps. For the case of arithmetic  families of Riemann surfaces we refer to the articles \cite{th1} and \cite{th2}.

We mention that arithmetic Riemann surfaces (also known as {\it Belyi curves}) have attracted much
attention ever since Grothendieck noticed, in his famous
Esquisse d'un Programme, interesting relations between them and bipartite graphs embedded in
a topological surface; see \cite{Gro}.

\s

As in the case of real Riemann surfaces, arithmetic Riemann surfaces among the Riemann surfaces in the family $\mathcal{F}_q$ can be easily identified.
\begin{theo} \label{arit} Let $\lambda \in \Omega.$ Then the following statements are equivalent:
\begin{enumerate}
\item[(a)] $S_{\lambda}$ is an arithmetic Riemann surface.
\item[(b)] $JS_{\lambda}$ is an arithmetic algebraic variety.
\item[(c)] $\lambda$ is an algebraic complex number.
\end{enumerate}
\end{theo}

\begin{proof} As in the proof of Theorem \ref{lreal}, the equivalence between the first two statements follows from \cite[Theorem 1.1]{Milne}.

 We denote by $\mathcal{O}_{\lambda}$ the Riemann orbifold given by the action of $G_{\lambda}$ on $S_{\lambda},$ and by $$\pi_{G_{\lambda}}: S_{\lambda} \to \mathcal{O}_{\lambda}$$ the associated covering map. 

Let us assume that $S_{\lambda}$ is arithmetic. Then there exists an algebraic model $S'_{\lambda}$ of $S_{\lambda}$ defined by equations whose coefficients belong to the field of the algebraic numbers. Let us denote by $G'_{\lambda}$ the automorphism group of $S'_{\lambda},$ by $\mathcal{O}'_{\lambda}$ the Riemann orbifold given by the action of $G'_{\lambda}$ on $S'_{\lambda},$ and by $\pi_{G'_{\lambda}}$ the associated covering map.

As a consequence of \cite[Proposition 3.3]{criterio1}, both each element of $G'_{\lambda}$ and the projection $\pi_{G'_{\lambda}}$ are algebraic (i.e. defined over $\bar{\mathbb{Q}}).$ In particular, the branch values of $\pi_{G'_{\lambda}}$ are also algebraic. Let $\mu_0, \mu_1, \mu_{\infty}$ and $\mu_{\lambda}$ denote these values, where $\mu_0, \mu_1, \mu_{\infty}$ are marked with $2$ and $\mu_{\lambda}$ is marked with $2q.$

Now, the existence of an isomorphism $f_{\lambda}: S_{\lambda} \to S'_{\lambda}$ guarantees the existence of an isomorphism $g_{\lambda}: \mathcal{O}_{\lambda} \to \mathcal{O}'_{\lambda}$ such that $\pi_{G'_{\lambda}} \circ f_{\lambda}=g_{\lambda} \circ \pi_{G_{\lambda}}.$ It follows that $g_{\lambda}$ is a M\"{o}bius transformation satisfying that  $$g_{\lambda}(\mu_{\lambda})=\lambda \hspace{0,3 cm} \mbox{and}  \hspace{0,3 cm} g_{\lambda}(\{\mu_{\infty}, \mu_0, \mu_1\})=\{\infty, 0, 1\}.$$

Thus, $$g_{\lambda}(z)=T \left( \tfrac{(\mu_1-\mu_{\infty})(z-\mu_0)}{(t_1-t_0)(z-\mu_{\infty})}\right)$$ for some $T \in \mathbb{G}$ as in \eqref{sim}, and therefore$$\lambda = T \left( \tfrac{(\mu_1-\mu_{\infty})(\mu_{\lambda}-\mu_0)}{(\mu_1-\mu_0)(\mu_{\lambda}-\mu_{\infty})} \right).$$

Finally, as each $T \in \mathbb{G}$ is defined over $\mathbb{Q}$ and the points $\mu_0, \mu_1, \mu_{\infty},\mu_{\lambda}$ are algebraic, we are in position to conclude that the complex number $\lambda$ must be algebraic.

The converse follows directly from Theorem \ref{modelo}, and the proof is done.
\end{proof}

\begin{coro} \label{aritdos} Let $\lambda \in \Omega.$  Then $JS_{\lambda}$ is an arithmetic algebraic variety admitting a group algebra decomposition in which each factor is arithmetic as well.
\end{coro}

\begin{proof} Following \cite[Theorem 4.4]{criterio1},  if $S$ is an arithmetic Riemann surface then any Riemann surface $S'$ for which there is a covering map $S \to S'$ is arithmetic as well. Thus, the result follows directly from Theorems \ref{theo2} and \ref{arit}. \end{proof}

\begin{rema} \mbox{}
\begin{enumerate}
\item[(a)] It is worth observing that Theorem \ref{arit} and Corollary \ref{aritdos} can be easily generalized from $\bar{\mathbb{Q}}$ to any algebraically closed subfield $k$ of the field of the complex numbers. 
\item[(b)] In addition, Corollary \ref{aritdos} can also be generalized from each $S_{\lambda}$ in $\mathcal{F}_q$ to any Riemann surface $S$ defined over $k$ whose Jacobian variety admit a group algebra decomposition in which every factor is isogenous to the Jacobian of a quotient of $S.$
\end{enumerate}
\end{rema}

\subsection{Riemann surfaces defined over the field of moduli}
The {\it field of moduli} $\mathcal{M}(S)$ of a compact Riemann surface $S$ is by definition the fixed field of the group $$\mathbb{I}({S})=\{\sigma \in \mbox{Gal}(\mathbb{C}): S^{\sigma} \cong S\}.$$ 

\begin{prop}
Let $\lambda \in \Omega.$ Then $$\mathbb{Q}(j(\lambda)) \le \mathcal{M}(S) \le \mathbb{Q}(\lambda)$$where $j$ denotes the invariant function for elliptic curves, in the Legendre form.
\end{prop}

\begin{proof} We recall that, as a consequence of Theorem \ref{modelo} and Proposition \ref{eli}, $$(S_{\lambda})^{\sigma}=S_{\sigma(\lambda)} \hspace{0,3 cm} \mbox{and} \hspace{0,3 cm} (E_{\lambda})^{\sigma} = E_{\sigma(\lambda)}$$for all $\sigma \in \mbox{Gal}(\mathbb{C}),$ where $E_{\lambda}=S_{\lambda}/ \langle r^{-2}, sr^{-1} \rangle.$ 

 Now, if $\sigma \in \mathbb{I}({S})$ then there is an isomorphism $S_{\lambda} \to S_{\sigma(\lambda)}$ which induces an isomorphism $E_{\lambda} \to  E_{\sigma(\lambda)}.$ In particular, $$j(\lambda)=j(\sigma(\lambda))=\sigma(j(\lambda))$$showing that $\sigma \in \mbox{Gal}(\mathbb{C}/\mathbb{Q}(j(\lambda)));$ it follows that $\mathbb{Q}(j(\lambda)) \le \mathcal{M}(S)$. 

The other inclusion follows from Theorem \ref{modelo}, and from the fact that the field of moduli is contained in every field of definition. The proof is done.
\end{proof}

Weil in \cite{Weil} provided necessary conditions for $S$ to admit its field of moduli as a field of definition; these conditions hold trivially if $S$ does not have non-trivial automorphisms. On the other extreme, following \cite{Wolfart}, if $S/{\rm Aut}(S)$ is an orbifold with signature of type $(a,b,c)$ then $S$ can be defined over its field of moduli. 

By results of D\`ebes-Emsalem \cite{DE} (see also Hammer-Herrlich \cite{HH}) there is a field of definition of $S$ which is an extension of finite degree of its field of moduli.

\s

In general, the determination of whether the field of moduli is a field of definition is a difficult task; see, for example \cite{Earle}, \cite{Hid}, \cite{yo1}, \cite{yo2} and \cite{Shimura}. By contrast, in the hyperelliptic case it is possible to decide, in a very simple way, if the field of moduli is a field of definition. In fact, following \cite{hugg}, if the reduced automorphism group of a hyperelliptic Riemann surface is not cyclic, then it can be defined over its field of moduli. It follows immediately the following:

\begin{prop}  Let $\lambda \in \Omega.$ 
The field of moduli of $S_{\lambda}$ is a field of definition for $S_{\lambda}$ and for $JS_{\lambda}.$  
\end{prop}

\section{A three-dimensional family of ppavs with $\mathbf{D}_{10}$-action} \label{s6}

Let $S$ be a compact Riemann surface of genus $g \ge 2,$ and let $$JS=(\mathscr{H}^{1,0}(S, \mathbb{C}))^*/H_{1}(S, \mathbb{Z})$$be its Jacobian variety. We recall that, after fixing a symplectic basis of $H_{1}(S, \mathbb{Z})$, both a period matrix $(I_g \, Z_S)$ with $Z_S \in \mathscr{H}_g$ for $JS,$ and a rational representation of $L_S:=\mbox{End}_{\mathbb{Q}}(JS)$ are determined, up to equivalence. 

If $S$ is hyperelliptic, then the symplectic representation  $$\rho_r : G \to \mbox{Sp}(2g, \mathbb{Z})$$of the automorphism group $G$ of $S$ induces an isomorphism $$G \cong \mathcal{G}:= \{ R \in \mbox{Sp}(2g, \mathbb{Z}): R \cdot Z_S= Z_S\}.$$

We can now consider the complex submanifold of $\mathscr{H}_g$ $$\mathscr{H}_g(G)=\{ Z \in \mathscr{H}_g : R \cdot Z = Z \mbox{ for all } R \in \mathcal{G} \}$$consisting of those period matrices $Z$ representing ppavs of dimension $g$ admitting the given action of $G.$ Clearly, $Z_S \in \mathscr{H}_g(G).$

\s

In the case of the action of $\mathbf{D}_{10}$ on the Riemann surfaces in family $\mathcal{F}_5,$ we can be much more explicit.

\begin{theo} \label{invat}Consider the action of $\mathbf{D}_{10}$  with generating vector $\sigma_0.$  

There exists a three-dimensional family $\mathcal{A}_5(\mathbf{D}_{10})$ of principally polarized abelian varieties of dimension five admitting the
given group action; it is given by the period matrices in $\mathscr{H}_5$ of the following form:
\begin{equation} \label{inva}
\left(\begin{smallmatrix}
2(u + v + u)   &  -w - u                      & -2v              &  -v - w - u    &  -v + u  \\
-w - u                   & -v - \tfrac{1}{2}w+\tfrac{5}{4}u     & v-\tfrac{1}{2}u      &  w + \tfrac{1}{2}u       &  v - u     \\
-2v                          & v - \tfrac{1}{2}u                   &  u                 &  v                      & w  \\
-v - w - u         & w + \tfrac{1}{2}u                  & v                  &  u                      &  -w  \\
-v +u                 & v - u                        & w                   &  - w                &    2(u-v-w)
\end{smallmatrix}\right) 
\end{equation}for complex numbers $u, v$ and $w.$

Furthermore, $\mathcal{A}_5({\mathbf{D}_{10}})$ contains the one-dimensional family $\mathcal{F}_5.$ 
\end{theo}

\begin{proof} The proof is based on the results and routines in \cite{poligono} (implemented in the open source computer algebra system SAGE). 

By constructing a family of very special hyperbolic polygons that uniformize Riemann surfaces with a given group action, it was implemented, among others, routines to determine a symplectic representation of the group, and after that, those  matrices which are invariant.

We consider the generating vector $\sigma_0=(s, sr^{-2}, r^5, r^7)$ of $G=\mathbf{D}_{10}.$ By applying the routine {\it P.symplectic.generators}, we obtain that, if $\rho$ denotes the symplectic representation of $G,$ then $$\rho(r)=\mbox{diag}(R, (R^t)^{-1}) \hspace{0.2 cm} \mbox{ and }\hspace{0.2 cm}\rho(s)=\mbox{diag}(S, S^t),$$where  \begin{displaymath}
R=\left(\begin{smallmatrix}
-1  &   0   &  1  &  -1   &  1  \\
1   &   1   &  0  &    1  &  0  \\
0   &   0   &  0  &    -1& 0  \\
1   &   0   &  0  &    1& -1 \\
0   &  -2   &  0  &  1 & -1
\end{smallmatrix}\right) \hspace{0.2 cm} \mbox{ and }\hspace{0.2 cm} S=\left(\begin{smallmatrix}
-1  &   0   &  1  &  -1   &  1  \\
0   &   -1   &  -1  &    1  &  -1  \\
0   &   0   &  0  &    -1& 0  \\
0   &   0   & -1  &    0& 0 \\
0   &  0   &  0  &  0 & 1
\end{smallmatrix}\right)
\end{displaymath}

\s

The problem of finding those period matrices in $\mathscr{H}_5$ which are invariant under the given action involves solving a system of nonlinear equations in fifteen variables. If we apply the routine {\it P.moebius.invariant}, the desired form is obtained. 
\end{proof}

The automorphism group $G$ of $S$ can be canonically seen as a subgroup of $L_S.$ Thus, the variety $\mathscr{H}_g({G})$ contains the complex submanifold $\mathbb{H}(L_S)$ whose points are matrices representing ppavs containing $L_S$ in their endomorphism algebras; see \cite[Section 3]{Wolfart2} and also \cite[Sections 2 and 3]{Ru} for a more general context. This is called the {\it Shimura family} of $S$ and corresponds to a {\it special subvariety} of $\mathcal{A}_g$ (see \cite[Section 3]{Moonen} for a precise definition).

\begin{prop}
Let $\lambda \in \Omega.$ The dimension of the Shimura family of each Riemann surface $S_{\lambda}$ in $\mathcal{F}_q$ is $\tfrac{q+1}{2}.$
\end{prop}

\begin{proof} 
Following the results proved in \cite{frediani} and Serre's formula \cite[Proposition 3]{Serre}, it can be seen that the dimension $N$ of the Shimura family of $S_{\lambda}$ is given by $$\tfrac{1}{8q} \Sigma_{g \in G} [\chi(g)^2+\chi(g^2)],$$where $\chi$ stands for the character of the analytic representation $\rho_a$ of $G.$ Clearly, this dimension does not depend on $\lambda;$ in fact, it only depends on the local monodromy of the action of $G$ on $S_{\lambda}$.

Now, by using the classically known Chevalley-Weil formula \cite{cw}, we obtain that $$\rho_a \cong W_4 \oplus W_5.$$

The character of $\rho_a$ is summarized in the following table: \begin{center}
\begin{tabular}{|c|c|c|c|c|c|c|}  \hline
rep. of conj. class &  $id$ & $s$ &  $sr$ & $r^q$ &  $r^{t}$\\ \hline
lenght  & $1$ & $q$ & $q$ & $1$ & 2 \\ \hline
character  & $q$ & $-1$ & $1$ & $-q$ & 0 \\ \hline
\end{tabular}
\end{center}where $1 \le t \le q-1.$ It follows that $$N=\tfrac{1}{8q}[(q^2+q) +(1+q)q+(1+q)q+(q^2+q)]=\tfrac{q+1}{2}.$$\end{proof}

Given a Riemann surface $S$, to provide an explicit description of the elements of $\mathbb{H}(L_S)$ seems to be a difficult task. However, as a simple consequence of Theorem \ref{invat}, we obtain the following direct corollary:

\begin{coro}
Each element of the Shimura family associated to every member of the family $\mathcal{F}_5$ admits a period matrix of the form \eqref{inva} for some $u,v,w \in \mathbb{C}.$ 
\end{coro}

\s
\s

%%%%%%%%%%%%%%%%%%%%

\end{document}